\documentclass[11pt, reqno]{amsart}

\usepackage{amsmath, amsthm, amscd, amsfonts, amssymb, graphicx, color}
\usepackage[bookmarksnumbered, colorlinks, plainpages]{hyperref}
\usepackage{dsfont}
\input{mathrsfs.sty}
\usepackage{subfigure}
 \makeatletter
\let\reftagform@=\tagform@
\def\tagform@#1{\maketag@@@{(\ignorespaces\textcolor{blue}{#1}\unskip\@@italiccorr)}}
\renewcommand{\eqref}[1]{\textup{\reftagform@{\ref{#1}}}}
\makeatother
\usepackage{hyperref}
\hypersetup{colorlinks=true, linkcolor=red, anchorcolor=green,
citecolor=cyan, urlcolor=red, filecolor=magenta, pdftoolbar=true}
\newcommand{\vertiii}[1]{{\left\vert\kern-0.25ex\left\vert\kern-0.25ex\left\vert #1
    \right\vert\kern-0.25ex\right\vert\kern-0.25ex\right\vert}}
\newtheorem{theorem}{Theorem}[section]
\newtheorem{lemma}[theorem]{Lemma}

\newtheorem{corollary}[theorem]{Corollary}
\theoremstyle{definition}

\theoremstyle{remark}
\newtheorem{remark}[theorem]{Remark}
\begin{document}
\setcounter{page}{1}

\title[Some Generalizations of Mercer inequality  ]{Some Generalizations of Mercer inequality and its operator extensions }

\author[M. Kian  and Z. Peymani]{Mohsen Kian \MakeLowercase{and}   Zainab  Peymani Mazraj}

\address{Department of Mathematics, University of Bojnord, P.O. Box
1339, Bojnord 94531, Iran}
\email{kian@ub.ac.ir }
\email{zpeymani168@gmail.com}



\begin{abstract}
We study the Mercer inequality and its operator extension for  superquadratic functions.  In particular, we give a more general form of the Mercer inequality by replacing some constants by positive operators.  As some consequences, our results produce a Jensen operator inequality for superquadratic functions.  Moreover, we present  some Mercer inequalities of Hermite--Hadamard's type.

\end{abstract}

\maketitle

\section{Introduction and Preliminaries}
Mercer \cite{Mercer} proved a variant of the Jensen inequality for  convex functions as follows: If $f:[m,M]\to\mathbb{R}$ is a convex function, then
\begin{align}\label{mercer}
f\left( M+m-\sum_{j=1}^{n}\lambda_j x_j\right)\leq f(M)+f(m)-\sum_{j=1}^{n}\lambda_j f(x_j)
\end{align}
for all $x_j\in[m,M]$ and all $\lambda_j\in[0,1]$ with $\sum_{j=1}^{n}\lambda_j=1$. An operator extension of \eqref{mercer} has been presented in \cite{MPP}:
\begin{align}\label{op-mercer}
f\left( M+m-\sum_{j=1}^{n}\Phi_j(A_j)\right)\leq f(M)+f(m)-\sum_{j=1}^{n}\Phi_j(f(A_j))
\end{align}
in which $(\Phi_1,\dots,\Phi_n)$ is a unital tuple of positive linear maps on $\mathcal{B}(\mathcal{H})$ and $A_j$'s are self-adjoint operators with spectra in $[m,M]$.
Here we denote by $\mathcal{B}(\mathcal{H})$ the algebra of all bounded linear operators on a Hilbert space $\mathcal{H}$. When $f:[m,M]\to\mathbb{R}$ is  a continuous function and    $A$ is a  self-adjoint operator  with spectrum in $[m,M]$, the operator $f(A)$ is defined by  the well-known Gelfand's mapping. This is called the continuous functional calculus.

Utilising the famous Hermite--Hadamard inequality, we presented  \cite{KM} a variant of the operator Mercer inequality \eqref{op-mercer}. Some reverse Mercer operator inequalities have been given in \cite{An2}. Some majorization type Mercer inequalities were established in \cite{niz2,niz}. A continuous  form of \eqref{op-mercer} can be found in \cite{IPM}.

 In \cite{MMK}, we extended the operator Mercer inequality \eqref{op-mercer}  by replacing the scalars $m$ and $M$ by operators and showed that with some conditions on the spectra of operators, the inequality
 \begin{align}\label{mmk}
   f(\Phi(C))+f(\Phi(B))\leq \Phi(f(A))+\Phi(f(D))
 \end{align}
holds.

Superquadratic functions are  introduced in \cite{ajs} as modifications of convex functions. Since then, this class of functions has been utilised to improve many results concerning convex functions.   A function $f:[0,\infty)\to\mathbb{R}$  is called superquadratic if for all $x\geq0$ there exists a constant $C_x\in\mathbb{R}$  such that
  \begin{eqnarray*}
    f(y)\geq f(x)+C_x(y-x)+f(|y-x|)
  \end{eqnarray*}
for all $y\geq0$. These functions enjoy a Jensen type inequality as
\begin{align}\label{j-s}
f(\lambda x+(1-\lambda)y)\leq \lambda f(x)+(1-\lambda)f(y)-\lambda f((1-\lambda)|x-y|)-(1-\lambda)f(\lambda|x-y|)
\end{align}
for all $x,y\geq0$ and $\lambda\in[0,1]$, see \cite{ajs}.

The authors of \cite{BMJ} proved a variant of the operator Mercer inequality \eqref{op-mercer} for superquadratic functions: If $f:[0,\infty)\to\mathbb{R}$ is a continuous superquadratic function and $m,M$ are positive scalars, then
\begin{align}\label{op-me-s}
  f(M+m-\Phi(A))-\beta(\Phi(A))\leq f(m)+f(M)-\Phi(f(A))-\Phi(\beta(A))
\end{align}
holds for every positive operator $A\in\mathcal{B}(\mathcal{H})$   with spectrum in $[m,M]$ and every positive linear map $\Phi$ on $\mathcal{B}(\mathcal{H})$, when we set the notation
\begin{align}\label{beta}
\beta(t) = \dfrac{t -m}{M - m}f(M - t) +\dfrac{M-t}{M - m}f(t - m),\qquad (t\in[m,M]).
  \end{align}
 Recall that an operator $A\in\mathcal{B}(\mathcal{H})$ is called positive if $\langle A\eta,\eta\rangle\geq0$ for every $\eta\in \mathcal{H}$.

The importance of \eqref{op-mercer}, \eqref{mmk} and \eqref{op-me-s} is that they are available without a restrictive condition on the function $f$. It is known that if $f:[m,M]\to\mathbb{R}$ is operator convex, then the Jensen operator inequality $f(\Phi(A))\leq \Phi(f(A))$ holds for every self-adjoint operator $A\in\mathcal{B}(\mathcal{H})$ with spectrum in $[m,M]$ and every unital positive linear map $\Phi$ on $\mathcal{B}(\mathcal{H})$. However, if $f$ is convex, but not operator convex, then this inequality does not hold in general. However, \eqref{op-mercer} is valid for every convex function (\eqref{op-me-s} is valid for every superquadratic function). Some related results can be found in \cite{An}.

In this paper, we  study the Mercer inequality and its operator extension for superquadratic functions. In particular, we extend \eqref{op-me-s} by replacing scalars $m,M$ by operators. As applications, a Jensen operator inequality has been presented for superquadratic functions.  Moreover, we present a Mercer inequality of   Hemite--Hadamard's type.

\section{Main Result}

We begin by presenting a Mercer inequality of   Hemite--Hadamard's type for superquadratic functions. We need the following lemma. For simplicity, we use the notation $x\nabla_\lambda y$ for the $\lambda$-weighted arithmetic mean of $x$ and $y$.
\begin{lemma}\label{lem1}
 Let $0\leq m < M$ and let $f : [0,\infty)\to \mathbb{R}$ be  a  superquadratic function.  Then
  \begin{align*}
&f(m + M -  x\nabla_\lambda y)+2 \beta(x)\nabla_\lambda \beta(y)\\
&\qquad\leqslant f(m) + f(M)
  -f(x)\nabla_\lambda f(y)
  - f((1 - \lambda)|x - y|)\nabla_\lambda f(\lambda|x - y|))
  \end{align*}
for all $x, y \in  [m, M]$ and every $\lambda\in[0,1]$.
  \end{lemma}
    \begin{proof}
    For any  $x \in [m, M]$,  we  put $y = m + M - x$ so that $y + x = m + M$. There  exists
$\lambda\in [0, 1]$ such that $y = \lambda m + (1- \lambda)M$. Since $f$ is superquadratic, we have
 \begin{align}\label{w2}
&f(M + m - x)=f(\lambda m + (1 -\lambda)M)\notag\\
&\leqslant \lambda f(m) + (1 - \lambda)f(M) - \lambda f((1 - \lambda)|m - M|) - (1 - \lambda)f(\lambda|m - M|)\\
&=f(m) + f(M) -\Big( \lambda f(M) + (1 - \lambda)f(m) + \lambda f((1 - \lambda)|M - m|)+ (1 - \lambda)f(\lambda|M- m|) \Big)\notag.
  \end{align}
On the other hand, $x=m+M-y= \lambda M + (1- \lambda)m$ and so we have
 \begin{align}\label{w1}
f(x)&=f(\lambda M + (1- \lambda)m)\notag\\
&\leq \lambda f(M) + (1 - \lambda)f(m) - \lambda f((1 - \lambda)|M - m|)- (1 - \lambda)f(\lambda|M- m|).
  \end{align}
  It follows from \eqref{w2} and \eqref{w1} that
  \begin{align}\label{w3}
f(M + m - x)\leq f(m) + f(M) -f(x)-2\left(\lambda f((1 - \lambda)|M - m|)+(1 - \lambda)f(\lambda|M- m|)\right).
  \end{align}
  Applying \eqref{w3}
  with $\lambda =\dfrac{x -m}{M - m}$  we obtain
\begin{equation}\label{eq1.0}
f(M + m - x) \leqslant f(m) + f(M) - f(x) - 2\beta(x),
\end{equation}
in which
\begin{align}\label{beta}
\beta(x) = \dfrac{x -m}{M - m}f(M - x) +\dfrac{M-x}{M - m}f(x - m),\qquad (x\in[m,M]).
  \end{align}
  Now, for every  $x, y \in [m, M]$ and every $\lambda\in[0,1]$ we have
    \begin{align}\label{eq2.0}
 f (M + m - (\lambda x + (1 - \lambda)y))  &= f(\lambda(m + M - x) + (1 - \lambda)(m + M - y))\nonumber\\
&\leq  \lambda f(m + M - x) + (1 - \lambda)f(m + M - y)\\
&\quad - \lambda f((1 - \lambda)|x - y|) - (1 - \lambda)f(\lambda|x - y|) \nonumber
  \end{align}
  as $f$ is superquadratic. Now, by applying  \eqref{eq1.0}  we get
  \begin{align}\label{eq3.0}
&\lambda f(m + M - x) + (1 - \lambda)f(m + M - y)\nonumber\\
&\quad \leqslant  \lambda(f(m) + f(M) - f(x) - 2\beta(x))
+ (1 - \lambda) (f(m) + f(M) - f(y) - 2\beta( y))\nonumber\\
&\quad=f(m) + f(M) - (\lambda f(x) + (1 - \lambda)f(y))- 2 (\lambda \beta( x) + (1 - \lambda)\beta(y)).
  \end{align}
 It follows from  (\ref{eq2.0}), (\ref{eq3.0}) that
    \begin{align*}
&f(m + M - (\lambda x + (1 - \lambda)y)\\
&\quad \leq f(m) + f(M) - (\lambda f(x) + (1 - \lambda)f(y)) - 2(\lambda \beta(x) + (1 - \lambda)\beta(y))\\
&\qquad - (\lambda f((1 - \lambda)|x - y|)
+(1 - \lambda)f(\lambda |x - y|)),
   \end{align*}
as required.
    \end{proof}
    %--------------------------------------------------------------
 In the next result, we present a Mercer inequality of Hermite--Hadamard's type for
 superquadratic functions.  The reader may compare it to \cite[Theorem 2.1]{KM}.
 \begin{theorem}\label{th-main}
      Let $0\leq m < M$ and let $f$ be a superquadratic function on $[m, M]$. Then
   {\small  \begin{align}\label{main1}
    \begin{split}
    &f\left(m+M-\dfrac{x+y}{2}\right)+2\int_{0}^{1/2}f\left(u|x-y|\right)\mathrm{d}u\\
    &\qquad\leq \frac{1}{y-x}\int_{x}^{y}f(m+M-u)\mathrm{d}u\\
    &\qquad\leq f(m) + f(M)-\frac{f(x)+f(y)}{2}-(\beta(x)+\beta(y))
   -2\int_{0}^{1}(1 - u)f(u |x - y|)\mathrm{d}u
     \end{split}
   \end{align}  }
   and
  {\small  \begin{align}\label{main2}
    \begin{split}
    &f\left(m+M-\dfrac{x+y}{2}\right)+  2\int_{0}^{1/2}f(u|x-y|)\mathrm{d}u\\
    &\qquad\leq f(m) + f(M)-\dfrac{1}{y-x}\int  _ x^y
(f(u)+2\beta (u))\mathrm{d}u \\
    &\qquad\leq f(m) + f(M)-f\left(\frac{x+y}{2}\right)-\dfrac{2}{y-x}\int  _ x^y\beta (u)\mathrm{d}u -2\int_{0}^{1/2}f(u|x-y|)\mathrm{d}u
     \end{split}
   \end{align}  }
   for all $x,y\in[m,M]$.
 \end{theorem}

 \begin{proof}
   Assume that $x,y\in[m,M]$ and put $a=M+m-x$ and $b=m+M-y$. Without loss of generality we assume that $x<y$.  Then
   \begin{align*}
    f\left(m+M-\dfrac{x+y}{2}\right)&=f\left(\dfrac{(ta + (1-t)b)+((1-t)a+tb)}{2}\right)\\
    &\leq
    \frac{f(ta + (1-t)b)+f((1-t)a+tb)}{2}-f\left(\left|\frac{2t-1}{2}\right||a-b|\right)
   \end{align*}
   since $f$ is superquadratic. Integrating both sides of the above inequality with respect to $t$ over $[0,1]$ yields
     \begin{align}\label{q20}
    f\left(m+M-\dfrac{x+y}{2}\right)
    &\leq \int_{0}^{1}  f(m+M-((1-t)x +t y))\mathrm{d}t - \int_{0}^{1} f\left(\left|\frac{2t-1}{2}\right||a-b|\right)\mathrm{d}t\\
    &=\frac{1}{y-x}\int_{x}^{y}f(m+M-u)\mathrm{d}u- 2 \int_{0}^{1/2}f\left(u|x-y|\right)\mathrm{d}u,\nonumber
   \end{align}
   where in the las equality, we employ change of variables in both integrals. On the other hand, it follows from Lemma~\ref{lem1} that
      \begin{align}\label{q21}
f(m+M-(tx + (1-t)y))&\leq f(m) + f(M) - (t f(x) + (1 - t)f(y))\\
&\,\, - 2(t \beta (x) + (1 - t)\beta(y))
  - (t f((1 - t)|x - y|)
+(1 - t)f(t |x - y|)).\nonumber
   \end{align}
  Noting that
   \begin{align*}
   \int_{0}^{1}t f((1 - t)|x - y|)\mathrm{d}t&=\int_{0}^{1}(1 - t)f(t |x - y|)\mathrm{d}t
   \end{align*}
    and
   integrating both sides of \eqref{q21} in terms of $t$ over $[0,1]$ we get
   \begin{align}\label{q22}
    \int_{0}^{1}  &f(m+M-(tx + (1-t)y))\mathrm{d}t \nonumber\\
    &\leq
     f(m) + f(M)-\frac{f(x)+f(y)}{2}-(\beta(x)+\beta(y)) -2\int_{0}^{1}(1 - t)f(t |x - y|)\mathrm{d}t.
   \end{align}
   Combining \eqref{q20} and \eqref{q22} we reach \eqref{main1}.

 Next,  it follows from Lemma~\ref{lem1}  that
  \begin{equation}\label{eq6.0}
f\left(m+M-\dfrac{a + b}{2}\right)\leqslant f(m)+f(M)-\dfrac{f(a) + f(b)}{2} -(\beta (a)+\beta (b))-f\left(\left|\dfrac{a - b}{2}\right|\right)
\end{equation}
holds for all $a, b \in  [m, M]$. Let $t \in  [0, 1]$ and $x, y \in  [m, M]$. Replacing $a$ and $b$,
respectively,  by $tx + (1 - t)y$ and $(1 - t)x + ty$ in (\ref{eq6.0}), we obtain
   \begin{align*}%\label{eq7.0}
&f\Big (m + M-\dfrac{tx + (1 - t)y + (1 - t)x + ty}{2}\Big )\nonumber\\
&\leqslant f(m) + f(M)  -\dfrac{f(tx + (1 - t)y) + f((1 - t)x + ty)}{2}      \\
& \quad - (\beta (tx + (1 - t)y) + \beta ((1 - t)x + ty))- f\left (\left |\dfrac{tx + (1 - t)y - (1 - t)x + ty}{2}\right |\right ) \nonumber
  \end{align*}
 or equivalently, we get
 \begin{align}\label{eq8.0}
 f\left (m + M-\dfrac{x + y}{2}\right ) \leqslant &f(m) + f(M)-\dfrac{f(tx + (1 -t)y) + f((1 -t)x + ty)}{2}\\
&-(\beta (tx + (1 -t)y) + \beta ((1 -t)x + ty)) -f\left (\dfrac{|1 -2t||x -y|}{2}\right ).\nonumber
  \end{align}
Note that
\begin{equation}\label{eq10.0}
\int  _ 0^1f(tx + (1 -t)y)\mathrm{d}t = \int  _ 0^1 f((1 -t)x + ty)\mathrm{d}t = \dfrac{1}{x-y} \int  _ x^y f(u)\mathrm{d}u,
\end{equation}
and
\begin{equation}\label{eq11.0}
\int  _ 0^1\beta(tx + (1 -t)y)\mathrm{d}t = \int  _ 0^1 \beta((1 -t)x + ty)\mathrm{d}t = \dfrac{1}{x-y} \int  _ x^y \beta(u)\mathrm{d}u.
\end{equation}
Consequently,   the first inequality of \eqref{main2} follows by integrating both sides of (\ref{eq8.0}) over  $t\in[0,1]$. To obtain the second inequality, we write
\begin{align*}
  f\left(\frac{x+y}{2}\right)&=f\left(\frac{(tx+(1-t)y)+((1-t)x+ty)}{2}\right)\\
  &\leq \frac{f(tx+(1-t)y)+f((1-t)x+ty)}{2}-f\left(\left|\frac{1-2t}{2}\right||x-y|\right).
\end{align*}
 Integrating both sides with respect to $t$ over $[0,1]$ we get
 \begin{align*}
  f\left(\frac{x+y}{2}\right)\leq \frac{1}{y-x}\int_x^yf\left(u \right)\mathrm{d}u
  -2 \int_0^{1/2} f(u|x-y|)\mathrm{d}u.
\end{align*}
This completes the proof.

 \end{proof}
%%----------------------------------------------------------
 In a particular case, the Mercer type inequality presented in Theorem~\ref{th-main}, concludes a Hermite--Hadamard inequality for superquadratic functions. The next corollary follows from Theorem~\ref{th-main}, when we consider $m=x$ and $M=y$.
%%------------------------------------------------
\begin{corollary}
 If $f:[0,\infty)\to\mathbb{R}$ is a  superquadratic function, then
   {\small  \begin{align}\label{main1-1}
    \begin{split}
     f\left(\dfrac{x+y}{2}\right)
     +2\int_{0}^{1/2}f\left(u|x-y|\right)\mathrm{d}u
    & \leq \frac{1}{y-x}\int_{x}^{y}f(t)\mathrm{d}t\\
    & \leq  \frac{f(x)+f(y)}{2}-2f(0)
   -2\int_{0}^{1}(1 - u)f(u |x - y|)\mathrm{d}u
     \end{split}
   \end{align}  }
   for all   $0\leq x<y$.
\end{corollary}

%%----------------------------------------------------
The power functions $f(t)=t^p$ and $g(t)=-t^q$ are superquadratic, when $p\geq2$ and $q\in[1,2]$. Hence, the next result follows.
\begin{corollary}\label{co1}
Let $0\leq m<M$ and let $x,y\in[m,M]$. If $p\geq2$, then
   {\small  \begin{align}\label{co-q1}
    \begin{split}
    &\left(m+M-\dfrac{x+y}{2}\right)^p
   +\frac{1}{2^p(p+1)}|x-y|^p \\
   &\qquad\leq
   \frac{(M+m-x)^{p+1}-(M+m-y)^{p+1}}{(p+1)(y-x)} \\
    &\qquad\leq m^p+M^p-\frac{x^p+y^p}{2}-(\beta_p(x)+\beta_p(y))
   -2\frac{|x-y|^p}{(p+1)(p+2)},
     \end{split}
   \end{align}}
   in which $\beta_p(x)=\frac{(M-x)(x-m)}{M-m}
   \left((M-x)^{p-1}+(x-m)^{p-1}\right)$. If $p\in[1,2]$, then \eqref{co-q1} is reversed.
\end{corollary}

 Let $f(t)=t^2$. Then $f$ is superquadratic as well as subquadratic. This fact together  Corollary~\ref{co1}  produce  an  equation  as
   {\small  \begin{align}\label{co-q2}
    \begin{split}
    &\left(m+M-\dfrac{x+y}{2}\right)^2
   +\frac{1}{12}|x-y|^2 \\
   &\qquad=
   \frac{(M+m-x)^{3}-(M+m-x)^{3}}{3(y-x)} \\
    &\qquad= m^2+M^2-\frac{x^2+y^2}{2}-((M-x)(x-m)+(M-y)(y-m))
   -\frac{|x-y|^2}{6}.
     \end{split}
   \end{align}}

%%%-----------------------------------------------------

\bigskip
\begin{lemma}
  Let $f:[0,\infty)\to\mathbb{R}$ be a superquadratic function. Then
    \begin{align}\label{th2-q1}
    \begin{split}
     f(x)+f(y)\leq f(x+y)-2\frac{yf(x)+xf(y)}{x+y}
     \end{split}
   \end{align}
   for all $x,y\geq0$. In particular, if $f$ is positive, then $f$ is super-additive. If $f$ is nonpositive, then $-f$ is sub-additive.
\end{lemma}
\begin{proof}
  Let $x,y\geq0$. Since $f$ is superquadratic, it follows from \eqref{j-s} that
   \begin{align}\label{qp1}
    \begin{split}
     f(x)&= f\left(\frac{x}{x+y}(x+y)+\frac{y}{x+y}0\right)\\
     &\leq \frac{x}{x+y}f(x+y)+\frac{y}{x+y}f(0)-\frac{x}{x+y}f(y)-\frac{y}{x+y}f(x).
     \end{split}
   \end{align}
   Similarly, we have
    \begin{align}\label{qp2}
    \begin{split}
     f(y)\leq \frac{y}{x+y}f(x+y)+\frac{x}{x+y}f(0)-\frac{x}{x+y}f(y)-\frac{y}{x+y}f(x).
     \end{split}
   \end{align}
   Summing both sides of \eqref{qp1} and \eqref{qp2} and noticing that $f(0)\leq0$, we achieve \eqref{th2-q1}.
\end{proof}

%%--------------------------------------------------

\begin{lemma}
  Let $f:[0,\infty)\to\mathbb{R}$ be a superquadratic function and let $0\leq y_1\leq x_1\leq x_2\leq y_2$. If $x_1+x_2=y_1+y_2$, then
    \begin{align}\label{lm3-q1}
    \begin{split}
     f(x_1)+f(x_2)\leq f(y_1)+f(y_2)-2\frac{y_2-x_1}{y_2-y_1}f(x_1-y_1)-2\frac{x_1-y_1}{y_2-y_1}f(x_2-y_1).
     \end{split}
   \end{align}
\end{lemma}
\begin{proof}
  Applying \eqref{j-s} with $\lambda=\frac{y_2-x_1}{y_2-y_1}$ we obtain
     \begin{align}
    \begin{split}
     f(x_1)&=f\left(\frac{y_2-x_1}{y_2-y_1}y_1+\frac{x_1-y_1}{y_2-y_1}y_2\right)\\
     &\leq  \frac{y_2-x_1}{y_2-y_1}f(y_1)+\frac{x_1-y_1}{y_2-y_1}f(y_2)
     -\frac{y_2-x_1}{y_2-y_1}f(x_1-y_1)
     -\frac{x_1-y_1}{y_2-y_1}f(y_2-x_1).
     \end{split}
   \end{align}
   Similarly with $\lambda=\frac{y_2-x_2}{y_2-y_1}$ we get
       \begin{align}
    \begin{split}
     f(x_2)\leq  \frac{y_2-x_2}{y_2-y_1}f(y_1)+\frac{x_2-y_1}{y_2-y_1}f(y_2)
     -\frac{y_2-x_2}{y_2-y_1}f(x_2-y_1)
     -\frac{x_2-y_1}{y_2-y_1}f(y_2-x_2).
     \end{split}
   \end{align}
   The desired inequality now follows from summing two last inequalities.
\end{proof}

Now we present our main result. It is an operator extension of \eqref{lm3-q1}. It also gives a generalization of the operator  Mercer inequality for superquadratic functions, see \cite{BMJ}.
\begin{theorem}\label{th3}
  Let $f:[0,\infty)\to\mathbb{R}$ be a continuous superquadratic function. Let $A,B,C,D$ be positive operators on a Hilbert space $\mathcal{H}$ such that $A+D=B+C$ and $0\leq A\leq m I\leq B\leq C\leq MI\leq D$ for some positive scalars $m,M$. If $\Phi$ is a unital positive linear map on $\mathcal{B}(\mathcal{H})$, then
  \begin{align}\label{th3-q1}
  \begin{split}
&f(\Phi(B)) + f(\Phi(C)) + \beta(\Phi(B))+ \beta(\Phi(C))\\
&\,\,\leq \Phi(f(A))+\Phi(f(D)) -\Phi(f(m-A))-\Phi(f(D-M))+\frac{\Phi(A-D)+M-m}{M-m}f(M-m)
     \end{split}
   \end{align}
   in which $\beta(t)$ is defined by \eqref{beta}.
\end{theorem}
\begin{proof}
As $f$ is continuous, the function $\beta$ is continuous too. Moreover,  $\beta(t)=\beta(M+m-t)$ for every $t\in[0,\infty)$. Hence, we can apply the functional calculus to define $\beta(X)$ for every positive operator $X$. 

If $0\leq s \notin (m,M)$, then $s\in(0,m]\cup[M,\infty)$. First  we assume that $s\in[M,\infty)$ and we put $\mu=\frac{M-m}{s-m}\in[0,1]$. Applying \eqref{j-s} we obtain
       \begin{align*}
    \begin{split}
f(M)&=f\left(\mu  s+(1-\mu)m\right)\\
&\leq  \frac{M-m}{s-m} f(s)+ \frac{s-M}{s-m}f(m)-  \frac{M-m}{s-m}f(s-M)-\frac{s-M}{s-m}f(M-m)
     \end{split}
   \end{align*}
or equivalently,
       \begin{align}\label{qp7}
    \begin{split}
f(s)-f(s-M)+\frac{M-s}{M-m}f(M-m)\geq  \frac{M-s}{M-m} f(m)+ \frac{s-m}{M-m}f(M).
     \end{split}
   \end{align}
   If $s\in[0,m)$, then a similar argument yields
         \begin{align}\label{qp8}
    \begin{split}
f(s)-f(m-s)+\frac{s-m}{M-m}f(M-m)\geq  \frac{M-s}{M-m} f(m)+ \frac{s-m}{M-m}f(M).
     \end{split}
   \end{align}
As $A\leq mI$ and $D\geq MI$,  we can apply functional calculus to \eqref{qp7} and \eqref{qp8}, respectively,  with $s=D$ and $s=A$ to derive
    \begin{align*}
    \begin{split}
f(D) -f(D-M)+\frac{M-D}{M-m}f(M-m)\geq  \frac{M-D}{M-m} f(m)+ \frac{D-m}{M-m}f(M)
     \end{split}
   \end{align*}
   and
    \begin{align*}
    \begin{split}
f(A)-f(m-A)+\frac{A-m}{M-m}f(M-m)\geq  \frac{M-A}{M-m} f(m)+ \frac{A-m}{M-m}f(M).
     \end{split}
   \end{align*}
Applying the positive linear map $\Phi$ to both sides of the last two inequalities we reach
   \begin{align}\label{qp10}
  \begin{split}
\Phi(f(D)) -\Phi(f(D-M))+\frac{M-\Phi(D)}{M-m}f(M-m)\geq  \frac{M-\Phi(D)}{M-m} f(m)+ \frac{\Phi(D)-m}{M-m}f(M)
     \end{split}
   \end{align}
   and
    \begin{align}\label{qp11}
    \begin{split}
\Phi(f(A))-\Phi(f(m-A))+\frac{\Phi(A)-m}{M-m}f(M-m)\geq  \frac{M-\Phi(A)}{M-m} f(m)+ \frac{\Phi(A)-m}{M-m}f(M).
     \end{split}
   \end{align}

Next let $t\in[m,M]$ and put $\lambda=\frac{M-t}{M-m}$. It follows from \eqref{j-s} that
        \begin{align}\label{qp3}
    \begin{split}
f(t)=f\left(\lambda m+(1-\lambda)M\right)\leq  \frac{M-t}{M-m} f(m)+ \frac{t-m}{M-m}f(M)-\beta(t),
     \end{split}
   \end{align}
   where $\beta(t)$ is defined by \eqref{beta}. Since $\Phi$ is unital and positive, the spectra of operators $\Phi(B)$ and $\Phi(C)$ are contained in $[m,M]$. Accordingly, we can  apply the continuous functional calculus to \eqref{qp3} with $t=\Phi(B)$ and $t=\Phi(C)$ to get
    \begin{align}\label{qp4}
    \begin{split}
f(\Phi(B)) & + \beta(\Phi(B)) \leq  \frac{M-\Phi(B)}{M-m} f(m)+ \frac{\Phi(B)-m}{M-m}f(M)
     \end{split}
   \end{align}
   and
   \begin{align}\label{qp5}
    \begin{split}
f(\Phi(C)) & + \beta(\Phi(C)) \leq  \frac{M-\Phi(C)}{M-m} f(m)+ \frac{\Phi(C)-m}{M-m}f(M).
     \end{split}
   \end{align}

 Summing \eqref{qp4} and \eqref{qp5} we get
   \begin{align*}
    \begin{split}
&f(\Phi(B)) + f(\Phi(C)) + \beta(\Phi(B))+ \beta(\Phi(C))\\
&\,\,\leq  \frac{2M-\Phi(B+C)}{M-m} f(m)+ \frac{\Phi(B+C)-2m}{M-m}f(M) \\
&\,\, =\frac{2M-\Phi(A+D)}{M-m} f(m)+ \frac{\Phi(A+D)-2m}{M-m}f(M)\qquad(\mbox{by $A+D=B+C$})\\
&\,\,\leq \Phi(f(A))+\Phi(f(D)) -\Phi(f(m-A))-\Phi(f(D-M))+\frac{\Phi(A-D)+M-m}{M-m}f(M-m)
     \end{split}
   \end{align*}
where the last inequality follows from summing \eqref{qp10} and \eqref{qp11}.
This completes the proof.
\end{proof}

%-----------------------------------------------------------------------------------
The next corollary gives another variant  of \eqref{th3-q1}. We omit the proof as it is similar to the proof of Theorem~\ref{th3}.
\begin{corollary}\label{co-2}
  With the hypothesis as in Theorem~\ref{th3}:
  {\small\begin{align*}
  \begin{split}
&\Phi(f(B)) + f(\Phi(C)) + \Phi(\beta(B))+ \beta(\Phi(C))\\
&\,\,\leq \Phi(f(A))+f(\Phi(D)) -\Phi(f(m-A))-f(\Phi(D)-M)+\frac{\Phi(A-D)+M-m}{M-m}f(M-m).
     \end{split}
   \end{align*}}

\end{corollary}
%%%---------------------------------------------------------

As a consequence, the Jensen-Mercer operator  inequality for superquadratic functions holds:
\begin{corollary}\cite[Theorem 1]{BMJ}
    Let $f:[0,\infty)\to\mathbb{R}$ be a continuous superquadratic function and let $0<m\leq M$.   If $C$ is a positive operator, whose spectrum is contained in $[m,M]$, then
    \begin{align*}
  \begin{split}
f(M+m-\Phi(C)) + \beta(\Phi(C))+f(0)\leq f(m)+f(M)-\Phi(\beta(C)) -f(0).
     \end{split}
   \end{align*}

\end{corollary}
\begin{proof}
 Let $C$ be a positive operator with spectrum   in $[m,M]$.   Apply Corollary \ref{co-2} with $A=mI$, $B=(M+m)I-C$ and $D=MI$.
\end{proof}

%%---------------------------------------------------------------------
As another consequence, we have the following Jensen operator inequality.
\begin{corollary}\label{co-4}
    Let $f:[0,\infty)\to\mathbb{R}$ be a continuous superquadratic function and let $0<m\leq M$.    Then
    \begin{align*}
  \begin{split}
f\left(\frac{A+D}{2}\right)+\beta\left(\frac{A+D}{2}\right)&\leq \frac{f(A)+f(D)}{2}-\frac{f(m-A)+f(D-M)}{2}
     \end{split}
   \end{align*}
for all positive operators $A$ and $D$ satisfying $A\leq mI\leq \frac{A+D}{2}\leq MI\leq D$.
\end{corollary}

\begin{remark}
  If the superquadratic function $f$ is positive, then $f$ is convex and Corollary~\ref{co-4} provide an improvement of \cite[Corollary 2.7]{MMK}. For example, if $f(t)=t^p$ with $p\geq2$, then
     \begin{align}\label{co4-q1}
  \begin{split}
\left(\frac{A+D}{2}\right)^p+\beta\left(\frac{A+D}{2}\right)&\leq \frac{A^p+D^p}{2}-\frac{(m-A)^p+(D-M)^p}{2}
     \end{split}
   \end{align}
  holds for all positive operators $A$ and $D$ satisfying $A\leq mI\leq \frac{A+D}{2}\leq MI\leq D$. The existence of scalars $m,M$ are necessary in Corollary~\ref{co-4}. For example, it is known that the function $f(t)=t^3$ is not operator convex and so one can find positive operators $A$ and $D$ such that the operator
  $$\frac{A^3+D^3}{2}-\left(\frac{A+D}{2}\right)^3$$
  is not positive. Accordingly,  \eqref{co4-q1} does not hold in general, while the function $f(t)=t^3$ is superquadratic.

  Moreover, if $f$ is a non-positive superquadratic function, then Corollary~\ref{co-4} gives  a reverse of \cite[Corollary 2.7]{MMK}.

\end{remark}

\begin{remark}
  An operator version of \eqref{th2-q1} also follows from Theorem~\ref{th-main} as follows:

  \begin{align}\label{sub}
    f(B)+f(C)+\beta(B)+\beta(C)\leq f(B+C)-f(B+C-M)
  \end{align}
  for all positive operators $B,C$ satisfying $0<B,C\leq M\leq B+C$ with $M>0$. To see this apply \eqref{th3-q1} with $A=m=0$ and $D=B+C$ and note that $f(0)\leq0$ for every superquadratic function $f$. It is  known that (see e.g. \cite{kosem}) if $f:[0,\infty)\to[0,\infty)$ is an increasing convex function with $f(0)=0$, then
   \begin{align}\label{sub-norm}
   \|f(B)+f(C)\|\leq \|f(B+C)\|
  \end{align}
  for all positive operators $B$ and $C$ and every unitarily invariant norm $\|\cdot\|$. We note that every positive superquadratic function $f$ is convex and satisfies $f(0)=0$. Hence, inequality \eqref{sub} gives an stronger result than \eqref{sub-norm}. However, the existence of positive scalar $M$ with  $B,C\leq M\leq B+C$ is necessary for \eqref{sub}. We give an example of such operators. Let $f(t)=t^3$ and put
  $$B=\begin{bmatrix}
        2 & -1 \\
        -1 & 2
      \end{bmatrix}\quad \mbox{and}\quad C=\begin{bmatrix}
        3 & 0 \\
       0 & 2
      \end{bmatrix}$$
  so that $B,C\leq 3I\leq B+C$.  We calculate
  $$
  f(B)=\begin{bmatrix}
        14 & -13 \\
        -13 & 14
      \end{bmatrix}\quad \mbox{and}\quad f(C)=\begin{bmatrix}
        8 & 0 \\
       0 & 27
      \end{bmatrix}\quad \mbox{and}\quad \beta(B)=5/3\begin{bmatrix}
        1 & 1 \\
       1 & 1
      \end{bmatrix}
  $$
  and
   $$
  \beta(C)=\frac{1}{3}\begin{bmatrix}
        10 & 0 \\
        0 & 0
      \end{bmatrix}\quad \mbox{and}\quad f(B+C)=\begin{bmatrix}
       77  & -62\\
   -62  & 139
      \end{bmatrix}\quad \mbox{and}\quad f(B+C-M)=\begin{bmatrix}
        5  & -8\\
   -8  & 13
      \end{bmatrix}.
  $$
Accordingly, we have
$$
f(B)+f(C)+\beta(B)+\beta(C)\simeq\begin{bmatrix}
       27  & -11.33\\
   -11.33  &  42.67
      \end{bmatrix}
      \leq
      \begin{bmatrix}
       72  & -54\\
   -54  &  126
      \end{bmatrix}= f(B+C)-f(B+C-M).
$$
  \end{remark}

\bigskip

%%------------------------------------------------------------------------------------------

\end{document}